\documentclass{article}

\usepackage{amsmath,amssymb,amsthm,mathtools,url}
\usepackage[textsize=footnotesize, backgroundcolor=yellow!10, linecolor=gray!35]{todonotes}

\newcommand{\argmin}{\operatornamewithlimits{arg\,min}}
\newcommand{\norm}[1]{\left\| #1\right\|}

\theoremstyle{plain}
\newtheorem{thm}{Theorem}[section]
\newtheorem{prop}{Proposition}[section]
\newtheorem{lem}{Lemma}[section]

\theoremstyle{definition}
\newtheorem{dfn}{Definition}[section]
\newtheorem{rmk}{Remark}[section]
\newtheorem{exmp}{Example}[section]

\author{
	Clemens Kirisits\footnote{Johann Radon Institute for Computational and Applied Mathematics (RICAM),
Austrian Academy of Sciences, Linz, Austria},
	Otmar Scherzer\footnote{Computational Science Center, University of Vienna, and Johann Radon Institute for Computational and Applied Mathematics (RICAM), Austrian Academy of Sciences, Linz, Austria}}
\date{\today}
\title{Convergence rates for regularization functionals with polyconvex integrands}

\begin{document}
\maketitle

\begin{abstract}
Convergence rates results for variational regularization methods typically assume the regularization functional to be convex. While this assumption is natural for scalar-valued functions, it can be unnecessarily strong for vector-valued ones. In this paper we focus on regularization functionals with polyconvex integrands. Even though such functionals are nonconvex in general, it is possible to derive linear convergence rates with respect to a generalized Bregman distance, an idea introduced by Grasmair in 2010. As a case example we consider the image registration problem.

\end{abstract}

\section{Introduction}
In this paper we consider solving ill-posed operator equations of the form 
\begin{equation}\label{eq:operator}
	K(u) = v,
\end{equation}
using Tikhonov-type regularization, which consists in approximation of a solution of \eqref{eq:operator} by the minimizer of the functional 
\begin{equation}
 \label{eq:Tik_reg}
 \norm{K(u) - v}^2 + \alpha \mathcal{R}(u).
\end{equation}
Regularization theory is well-established when $\mathcal{R}$ is convex, and in particular when $\mathcal{R}(u)=\frac{1}{2} \norm{u-u_0}^2$. See \cite{EngHanNeu96,Gro93,TikArs77,SchGraGroHalLen09,SchuKalHofKaz12,TikGonSteYag95,TikLeoYag98} for instance. Convergence rates results have been developed in \cite{Fle10,FleHof10,FleHofMat11,HofKalPoeSch07,SchGraGroHalLen09,SchuKalHofKaz12} among others. For nonconvex regularization functionals $\mathcal{R}$, however, only few results are available in the literature \cite{BreLor09,Gra10b,Zar09}.

If the sought-for solution $u$ is scalar-valued, then convexity of $\mathcal{R}$ is a natural condition, because it is closely linked to weak lower semicontinuity of $\mathcal{R}$. Yet if $u:\Omega \subset \mathbb{R}^n \to \mathbb{R}^N$ is a \emph{vector-valued} function, then properties strictly weaker than convexity are enough to ensure weak lower semicontinuity. On the other hand, using a nonconvex $\mathcal{R}$ raises the question of how to obtain convergence rates, since the most common approach involves Bregman distances, which in turn require $\mathcal{R}$ to be subdifferentiable. The aim of this article is to develop convergence rates results for regularization functionals with polyconvex integrands.

A function $f: \mathbb{R}^{N\times n} \to \mathbb{R}$ is polyconvex, if $f(A)$ can be written as a convex 
function of all subdeterminants of $A$. John Ball introduced this notion in the context of nonlinear elasticity, 
where convex stored energy functions are known to be too restrictive physically \cite{Bal77}. However, what 
lends importance to polyconvex functions even outside the field of elasticity is the fact that they render
functionals of the form
\begin{equation}\label{eq:R(u)}
	\mathcal{R}(u)=\int_\Omega f(x,u(x),\nabla u (x))\, dx
\end{equation}
weakly lower semicontinuous in $W^{1,p}(\Omega,\mathbb{R}^N)$.

Recently, the merits of polyconvex functions have been exploited in the field of image processing, in particular for image registration 
\cite{BurModRut13,DroRum04,IglRumSch15_report}. Practical applications of registration models are numerous, one of the most prominent 
being medical imaging \cite{FisMod08,SotDavPar13}.
Registering two given images $I_1,I_2: \Omega\to\mathbb{R}$ means finding a deformation $u:\Omega \to \mathbb{R}^n$ 
such that 
\begin{equation}\label{eq:i1i2} 
	I_1 \circ u = I_2.
\end{equation}
The ill-posedness of this problem is typically overcome via variational regularization, that is, by minimizing a functional of the form
$$\mathcal{S}(I_1\circ u, I_2) + \mathcal{R}(u),$$
where $\mathcal{S}$ measures the similarity between $I_1\circ u$ and $I_2$. A regularization functional $\mathcal{R}$ with a polyconvex integrand can be a reasonable choice, if one models $I_1$ and $I_2$ as hyperelastic materials. However, in this case standard convergence rates results from regularization theory do not apply \cite{SchGraGroHalLen09,SchuKalHofKaz12}. The aim of this paper is to address this issue.

\paragraph{Outline.} The next section (Sec.\ \ref{sec:prelim}) introduces the most important concepts and fixes some notation. It consists of three parts. In the first part, Section \ref{sec:Bregman}, we introduce (generalized) Bregman distances. In Section \ref{sec:regularization}, we review standard results on convergence rates for variational regularization of inverse problems in a Banach space setting. Section \ref{sec:polyconvex}, briefly discusses polyconvex functions and their properties. Section \ref{sec:registration} considers the image registration problem with polyconvex regularization from an inverse problems point of view. It also contains a specific example where in spite of nonconvex regularization the standard convergence rates result as stated in Sec.~\ref{sec:regularization} applies. Finally, in Section \ref{sec:main} we define $W_{\!\mathrm{poly}}$-Bregman distances for functionals with polyconvex integrands and state the corresponding convergence rates result.

\section{Preliminiaries} \label{sec:prelim}

\subsection{Bregman distances} \label{sec:Bregman}
In this article $U$ always denotes a Banach space with dual $U^*$. The dual pairing between $u\in U$ and $u^* \in U^*$ is denoted by $\langle u^*,u\rangle_{U^*,U}$. There are two notable special cases. If $U = U^* = \mathbb{R}^{N\times n}$, we write $u \cdot u^* = \sum_{i=1}^N \sum_{j=1}^n u^*_{ij}u_{ij}$. In the case of Lebesgue spaces (of possibly matrix-valued functions), we use dual brackets without subscripts and write
$$ \langle u^*,u\rangle  = \int u^*(x) \cdot u(x) \, dx.$$
Let $\Omega \subset \mathbb{R}^n$ be an open set. If $U=W^{1,p}(\Omega, \mathbb{R}^N)$, then every element $u^*$ of $U^*$ can be identified with a pair $(u^*_0,u^*_1)\in L^{p^*}(\Omega,\mathbb{R}^N \times \mathbb{R}^{N\times n})$ acting on $u\in U$ as
\begin{equation*}
	\langle u^*,u \rangle_{U^*,U} = \langle u^*_0,u \rangle + \langle u^*_1, \nabla u \rangle.
\end{equation*}

Let $\mathcal{R}$ be a function defined on $U$ taking values in the extended reals $\mathbb{R}\cup \{\pm\infty\}$. Its \emph{effective domain} $\mathrm{dom} \, \mathcal{R}$ is the set $\{u\in U : \mathcal{R}(u) < +\infty\}$. The \emph{subdifferential} of $\mathcal{R}$ at $u\in U$ is defined as
\begin{equation*}
	\partial \mathcal{R}(u) =
	\begin{cases}
		\{u^* \in U^* : \mathcal{R}(v) \ge \mathcal{R}(u) + \langle u^*, v-u \rangle_{U^*,U} \text{ for all } v \in U\}, & \mathcal{R}(u) \in \mathbb{R} \\
		\emptyset, & \mathcal{R}(u) \notin \mathbb{R}.
	\end{cases}
\end{equation*}
Note that we have not assumed $\mathcal{R}$ to be convex. If $\partial \mathcal{R}(u)\neq \emptyset$, then $\mathcal{R}$ is said to be \emph{subdifferentiable} at $u$ and elements $u^*\in \partial \mathcal{R}(u)$ are called \emph{subgradients}. Recall Fermat's rule: A proper function $\mathcal{R}$ attains its minimum at $u\in U$, if and only if $0 \in \partial \mathcal{R}(u)$. Let $u\in \mathrm{dom}\,\mathcal{R}$ and $u^* \in \partial \mathcal{R}(u)$. The \emph{Bregman distance} associated to $\mathcal{R}$ at $(u,u^*)$ is defined as
$$ D_{u^*}(v;u) = \mathcal{R}(v) - \mathcal{R}(u) - \langle u^*,v - u \rangle_{U^*,U}.$$
The following lemma justifies the use of the Bregman distance as a similarity measure.
\begin{lem}
	The Bregman distance is nonnegative and satisfies $D_{u^*}(u;u) = 0$.
\end{lem}
The Bregman distance is only defined at points where $\mathcal{R}$ has a subgradient. For convex functions these points can be characterized easily. The first two of the following three lemmas are classical results on subdifferentiability of convex functions. The third one deals with the special case of integral functionals on Sobolev spaces.
\begin{lem}\label{thm:subdif1}
	Let $\mathcal{R}:U\to \mathbb{R}\cup \{\pm\infty\}$ be a convex function. If $\mathcal{R}$ is finite and continuous at one point $\bar u\in U$, then $\partial \mathcal{R}(u) \neq \emptyset$ for all $u\in \mathrm{int} \, \mathrm{dom} \, \mathcal{R}$.
\end{lem}
\begin{proof}
	See Proposition 5.2 in Chapter I of \cite{EkeTem99}.
\end{proof}
\begin{lem}\label{thm:subdif2}
	If $\mathcal{R}:U\to \mathbb{R}\cup \{\pm\infty\}$ is proper, convex and lower semicontinuous, then the set $\{u \in U : \partial \mathcal{R}(u) \neq \emptyset \}$ is dense in $\mathrm{dom}\, \mathcal{R}$.
\end{lem}
\begin{proof}
	See Corollary 6.2 in Chapter I of \cite{EkeTem99}.
\end{proof}
\begin{lem}\label{thm:subdif3}
Let $\Omega \subset \mathbb{R}^n$ be an open set and let
	$$f:\Omega \times \mathbb{R}^N \times \mathbb{R}^{N\times n} \to [0,+\infty]$$
be a Carath\'eodory function. Assume that, for almost every $x\in \Omega$, the map $(u,A) \mapsto f(x,u,A)$ is convex and differentiable throughout its effective domain. Let $p\in[1,\infty)$ and define the following functional on $ W^{1,p}(\Omega,\mathbb{R}^N)$
	$$ \mathcal{R}(v) = \int_\Omega f(x,v(x),\nabla v(x)) \, dx.$$
Denote by $\nabla_{u,A} f$ the gradient of $f$ with respect to its second and third variables. If $v\in \mathrm{dom}\, \mathcal{R}$ and the function
$$x \mapsto \nabla_{u,A} f (x,v(x),\nabla v(x))$$
lies in $L^{p^*}(\Omega,\mathbb{R}^N \times \mathbb{R}^{N\times n})$, then this function is a subgradient of $\mathcal{R}$ at $v$.
\end{lem}
\begin{proof}
	This is a direct consequence of Lemma 4.1 in Chapter X of \cite{EkeTem99}.
\end{proof}
\begin{dfn} Let $W$ be a family of real-valued functions defined on $U$. Following \cite{Gra10b,Sin97} we define the \emph{W-subdifferential} of $\mathcal{R}$ at $u\in U$ as
\begin{equation*}
	\partial_W \mathcal{R}(u) =
	\begin{cases}
		\{w \in W : \mathcal{R}(v) \ge \mathcal{R}(u) + w(v) - w(u) \text{ for all } v \in U\}, & \mathcal{R}(u) \in \mathbb{R} \\
		\emptyset, & \mathcal{R}(u) \notin \mathbb{R}.
	\end{cases}
\end{equation*}
For $w\in \partial_W \mathcal{R}(u)$ the corresponding $W$-\emph{Bregman distance} is given by
\begin{equation} \label{eq:wbregman}
	D^W_w(v;u) = \mathcal{R}(v) - \mathcal{R}(u) - w(v) + w(u).
\end{equation}
\end{dfn}
Clearly, the $U^*$-subdifferential and the $U^*$-Bregman distance coincide with their classical counterparts.
\begin{lem}
	The $W$-Bregman distance is nonnegative and satisfies $$D^W_w(u;u) = 0.$$
\end{lem}

\subsection{Variational regularization on Banach spaces} \label{sec:regularization}
Let $U$, $V$ be Banach spaces and $K: U \to V$. We consider the inverse problem of finding $u \in U$ such that
\begin{equation}\label{eq:ip}
	K(u) = v^\dagger.
\end{equation}
Exact data $v^\dagger \in \mathrm{ran}\,K \subset V$ are assumed to be available as noisy measurements $v^\delta \in V$ only, satisfying $\| v^\dagger - v^\delta \| \le \delta$ for some $\delta \ge 0$. Since such problems are ill-posed in general, regularization is needed for the approximate inversion of $K$. Variational regularization consists in minimization of a functional of the form
\begin{equation}
	u \mapsto \mathcal{T}_\alpha (u; v^\delta) = \| K(u) - v^\delta \|^q + \alpha \mathcal{R}(u),
\end{equation}
where $q \ge 1$, $\alpha > 0$, $\|\cdot\|$ denotes the norm on $V$ and $\mathcal{R} : U \to [0,+\infty]$ is such that $\mathrm{dom}\,\mathcal{R} \cap \mathcal{D}(K) \neq \emptyset$. This variational approach leads to a \emph{well-defined regularization method}, if it fulfils the following requirements.
\begin{description}
	\item[Existence:] $\mathcal{T}_\alpha(\cdot;v^\delta)$ has a minimizer $u_\alpha^\delta$ for every $v^\delta\in V$ and $\alpha>0$.
	\item[Stability:] The inversion $v^\delta\mapsto u_\alpha^\delta$ is continuous.
	\item[Convergence:] There exists a parameter choice rule $\alpha:\mathbb{R}_+ \to \mathbb{R}_+$ such that regularized solutions $ u_\alpha^\delta$ converge to a solution of \eqref{eq:ip} as $\delta \to 0$.
\end{description}
The last point in particular requires that the set of \emph{exact solutions} $K^{-1}(v^\dagger)$ be nonempty. The subset of \emph{$\mathcal{R}$-minimizing solutions}, given by
	$$ \argmin \{ \mathcal{R}(u) : u \in K^{-1}(v^\dagger) \},$$
is of importance for the following theorem summarizing well-definedness of variational regularization methods in Banach spaces (cf.\ Section 3.2 of \cite{SchGraGroHalLen09}).
\begin{thm}\label{thm:welldef}
	Endow the Banach spaces $U$ and $V$ with topologies weaker than the respective norm topologies. Assume that the following four statements hold with respect to these topologies:
	\begin{enumerate}
		\item The sublevel sets of $\mathcal{T}_\alpha(\cdot;v^\dagger)$ are sequentially precompact.
		\item $\| \cdot \|$ is sequentially lower semicontinuous.
		\item $\mathcal{R}$ is sequentially lower semicontinuous.
		\item The sublevel sets of $\mathcal{T}_\alpha(\cdot;v^\dagger)$ are sequentially closed and $K$ is sequentially continuous there.
	\end{enumerate}
	Then the functional $\mathcal{T}_\alpha(\cdot;v^\delta)$ has a minimum for all $\alpha>0$ and $v^\delta \in V$. Moreover minimization of $\mathcal{T}_\alpha$ is continuous in the following sense. Whenever $\|v_k-v^\delta\|\to 0$, then every sequence $(u_k)$, $u_k \in \argmin \mathcal{T}_\alpha(\cdot;v_k)$, has a converging subsequence and the limit of every such sequence is a minimizer of $\mathcal{T}_\alpha(\cdot;v^\delta)$. Assume, in addition, that
	\begin{enumerate}\setcounter{enumi}{4}
		\item there is an exact solution, i.e.~$v^\dagger \in \mathrm{ran}\, K$ and
		\item the parameter choice rule $\alpha : \mathbb{R}_+ \to \mathbb{R}_+$ satisfies $\alpha(\delta) \to 0$ and $\delta^q/\alpha(\delta) \to 0$ as $\delta \to 0$.
	\end{enumerate}
	Then, whenever $\delta_k \to 0$, every sequence $(u_k)$, $u_k \in \argmin \mathcal{T}_\alpha(\cdot;v^{\delta_k})$, has a converging subsequence and the limit of every such sequence is an $\mathcal{R}$-minimizing solution.
\end{thm}
\begin{rmk}
Note that convexity of $\mathcal{R}$ is not required for a well-defined regularization method.
\end{rmk}
In principle convergence of regularized solutions can be arbitrarily slow. Therefore it is useful to have a bound in terms of $\delta$ on the discrepancy between regularized and exact solution. In a Banach space setting a typical discrepancy measure is the Bregman distance associated to the regularization functional \cite{BurOsh04}. Concerning convergence rates for variational regularization in Banach spaces we have the following result (cf.\ \cite{HofKalPoeSch07} or Section 3.2 of \cite{SchGraGroHalLen09}).
\begin{thm}\label{thm:convrates}
	Let the assumptions of the previous theorem hold. In addition assume that $\mathcal{R}$ has a subgradient $u^*$ at an $\mathcal{R}$-minimizing solution $u^\dagger$ and that there are constants $\beta_1 \in [0,1)$ and $\bar\alpha, \beta_2, \rho > 0$ satisfying $\bar\alpha\mathcal{R}(u^\dagger) < \rho$ such that
	\begin{equation}\label{eq:sourcecon}
		\langle u^* , u^\dagger - u \rangle \le \beta_1 D_{u^*} (u;u^\dagger) + \beta_2 \| K(u) - v^\dagger\|
	\end{equation}
	holds for all $u$ with $\mathcal{T}_{\bar \alpha} (u;v^\dagger) \le \rho$.
	
	If $q>1$, assume $\alpha(\delta) \sim \delta^{q-1}$. Then
	\begin{equation*}
		D_{u^*} (u^\delta_{\alpha};u^\dagger) = O(\delta) \quad \text{and} \quad \|K(u^\delta_{\alpha}) - v^\delta\| 	= O(\delta).
	\end{equation*}
	
	If $q=1$, assume $\alpha(\delta) \sim \delta^\epsilon$ for $\epsilon \in [0,1)$. If $\epsilon = 0$, additionally assume that $0 < \alpha(\delta) \beta_2 < 1$. Then
	\begin{equation*}
		D_{u^*} (u^\delta_{\alpha};u^\dagger) = O(\delta^{1-\epsilon}) \quad \text{and} \quad \|K(u^\delta_{\alpha}) - v^\delta\| 	= O(\delta).
	\end{equation*}
\end{thm}
\begin{rmk}\label{rmk:applicability}
In contrast to \cite{HofKalPoeSch07,SchGraGroHalLen09} we have not assumed $\mathcal{R}$ to be convex and have not added any other assumption in its place. This does not change the theorem's validity. However, since standard results on subdifferentiability (Lemmas \ref{thm:subdif1}, \ref{thm:subdif2} and \ref{thm:subdif3}) require $\mathcal{R}$ to be convex, the condition $u^*\in\partial\mathcal{R}(u^\dagger)$ must be expected to be difficult to satisfy for a general nonconvex $\mathcal{R}$. Yet in Example \ref{exmp:main} below we construct a problem where it \emph{is} satisfied.
\end{rmk}
\begin{rmk}\label{rmk:perfect}
If the regularization functional is chosen ``perfectly", that is, it has a global minimizer that is also an exact solution, then condition \eqref{eq:sourcecon} is always satisfied: Assume that $\bar{u}^\dagger$ is such a solution, that is, $K(\bar{u}^\dagger) = v^\dagger$ and $0\in\partial \mathcal{R}(\bar{u}^\dagger)$. Then, with $u^*=0$ inequality \eqref{eq:sourcecon} becomes
	$$ 0 \le \beta_1 (\mathcal{R}(u) - \mathcal{R}(\bar{u}^\dagger)) + \beta_2 \| K(u) - v^\dagger\|,$$
which is true for all $u\in U$ and all nonnegative $\beta_1,\beta_2$.
\end{rmk}

\subsection{Polyconvex functions} \label{sec:polyconvex}
Let $N,n \in \mathbb{N}$ and let $N\wedge n =\min(N,n)$. For $A\in \mathbb{R}^{N\times n}$ and $1\le s \le N\wedge n$ denote by $\mathrm{adj}_{s}(A)$ the matrix consisting of all $s\times s$ minors of $A$. Note that $\mathrm{adj}_{1}(A)=A$ and $\mathrm{adj}_{s}(A) \in \mathbb{R}^{\sigma(s)}$, where $\sigma(s) = \big( \begin{smallmatrix} N \\ s \end{smallmatrix} \big) \big( \begin{smallmatrix} n \\ s \end{smallmatrix} \big)$. Set $\tau(N,n) = \sum_{s=1}^{N\wedge n} \sigma(s)$ and denote by $T:\mathbb{R}^{N \times n} \to \mathbb{R}^{\tau(N,n)}$ the function that maps a matrix to the vector containing all its minors, which with a slight abuse of notation can be written as
$$T(A) = (A,\mathrm{adj}_{2}(A),\ldots,\mathrm{adj}_{m}(A)).$$
A function $f:\mathbb{R}^{N\times n} \to \mathbb{R}\cup \{+\infty\}$ is called \emph{polyconvex}, if there exists a convex function $F:\mathbb{R}^{\tau(N,n)} \to \mathbb{R}\cup\{+\infty\}$ satisfying $f = F \circ T$. Notice that this $F$ is not unique in general. Clearly, every convex function is polyconvex. If $N=1$ or $n=1$, then also the converse holds. See Chapter 5 in \cite{Dac08} for more details on polyconvex functions.

In the last part of the paper we will make use of the following variant of the map $T$. Set $\tau_2(N,n) = \sum_{s=2}^{N\wedge n} \sigma(s)$. We denote by $T_2:\mathbb{R}^{N \times n} \to \mathbb{R}^{\tau_2(N,n)}$ the function defined by
$$T_2(A) = (\mathrm{adj}_{2}(A),\ldots,\mathrm{adj}_{m}(A)).$$

Weak lower semicontinuity in $W^{1,p}(\Omega, \mathbb{R}^N)$, $p>N\wedge n$, of functionals of the form \eqref{eq:R(u)} can be established by requiring the integrand to be polyconvex in its third argument. The following theorem is a special case of the more general Theorem 8.16 in \cite{Dac08}.
\begin{lem}\label{thm:lsc}
Let $\Omega \subset \mathbb{R}^n$ be a bounded open set with Lipschitz boundary and let
$$ F: \Omega \times \mathbb{R}^N \times \mathbb{R}^{\tau(N,n)} \to [0,+\infty]$$
be a Carath\'eodory function such that the map $A \mapsto F(x,u,A)$ is convex for almost every $x \in \Omega$ and every $u \in \mathbb{R}^N$. Then, for $p>N\wedge n$, the functional
$$ u \mapsto \int_\Omega F(x,u(x),T(\nabla u (x)))\, dx $$
is sequentially weakly lower semicontinuous in $W^{1,p}(\Omega, \mathbb{R}^N)$.
\end{lem}

\section{Image registration} \label{sec:registration}

In this section we treat the image registration problem from an inverse problems perspective. First, by applying Theorem \ref{thm:welldef} we show that minimization of
$$ \|I_2 \circ u - I_1 \|^q + \alpha \mathcal{R}(u), $$
where $\mathcal{R}$ is a first order functional with polyconvex integrand, constitutes a well-defined regularization method. Second, we highlight a particular situation where, in spite of $\mathcal{R}$ being nonconvex, Theorem \ref{thm:convrates} applies as well.

Let $\Omega \subset \mathbb{R}^n$ be a bounded open set with Lipschitz boundary. Given a target image $I_1:\Omega \to \mathbb{R}$ and a reference image $I_2:\Omega \to \mathbb{R}$ the model equation for the image registration problem reads
$$ I_2 \circ u = I_1,$$
where $u:\Omega \to \mathbb{R}^n$ is an unknown deformation of the image domain. We interpret this as a particular instance of the abstract operator equation \eqref{eq:ip}. Thus $K$ is the composition operator that sends every deformation $u$ to the deformed reference image $I_2\circ u = K(u)$. Note that in Section \ref{sec:regularization} we have implicitly assumed that the operator be known exactly. Therefore, $I_2$ is known exactly, whereas the exact target image $I_1^\dagger$, i.e.~the exact data, is available only as noisy measurements $I_1^\delta$.

\begin{prop}\label{thm:registration} Let $p>n$ and $q\ge 1$. Endow $U=W^{1,p}(\Omega,\mathbb{R}^n)$, with its weak and $V=L^q(\Omega)$ with its strong topology. Assume $I_2\in C^0(\bar{\Omega})$ and define the operator
$$ K: U \to V, \quad u \mapsto K(u) = I_2\circ u$$
with domain $\mathcal{D}(K) = \{u\in U : u(\Omega) \subset \bar{\Omega}\}$. Let
$$ F : \Omega \times \mathbb{R}^n \times \mathbb{R}^{\tau(n,n)} \to [0,+\infty]$$
be a Carath\'eodory function such that, for almost every $x\in \Omega$ and every $u\in \mathbb{R}^n$, the map $\xi \mapsto F(x,u,\xi)$ is convex and
\begin{equation}\label{eq:coercivity}
	F(x,u,T(A)) \ge |A|^p
\end{equation}
holds for every $A\in \mathbb{R}^{n\times n}$. For $u\in U$ define
$$ \mathcal{R}(u) = \int_\Omega F(x,u(x),T(\nabla u(x)))\, dx $$
and assume that $\mathrm{dom}\, \mathcal{R} \cap \mathcal{D}(K)$ is not empty. Then, minimization of
$$ \mathcal{T}_\alpha(u;I^\delta_1) = \|K(u) - I_1^\delta\|^q_{L^q(\Omega)} + \alpha \mathcal{R}(u), \quad \alpha > 0,$$
is a well-defined regularization method in the sense of Theorem \ref{thm:welldef}.
\end{prop}
\begin{proof}
We show that all the assumptions of Theorem \ref{thm:welldef} are satisfied.

Item 2 is obviously true.

Item 3 follows from Lemma \ref{thm:lsc}.

Concerning Item 1 let $\alpha,M >0$, $I_1 \in V$ and $(u_k)\subset U$ with $\mathcal{T}_\alpha(u_k;I_1) \le M$ for $k\ge 1$. Then, in particular, $(u_k) \subset \mathcal{D}(K)$ and therefore $u_k(\Omega) \subset \bar\Omega$ for all $k$. Since $\Omega$ is bounded, the sequence $(u_k)$ is bounded in $L^p(\Omega,\mathbb{R}^n)$. The lower bound \eqref{eq:coercivity} on $F$ yields boundedness of $(u_k)$ in $U$. Since $p>1$, $U$ is reflexive and $(u_k)$ has a weakly convergent subsequence. Thus, the sublevel sets of $\mathcal{T}_\alpha(\cdot;I_1)$ are weakly sequentially precompact.

Concerning Item 4 let again $\mathcal{T}_\alpha(u_k;I_1) \le M$ for $k\ge 1$ and assume $u_k \rightharpoonup \bar u$ for some $\bar u \in U$. The compact embedding of $U$ into $C^0(\bar \Omega,\mathbb{R}^n)$ implies that $u_k \to \bar u$ uniformly and $\bar u\in \mathcal{D}(K)$. Since $I_2 \in C^0(\bar \Omega)$, the sequence $(I_2\circ u_k)$ converges uniformly to $I_2\circ \bar u$ and, because $\Omega$ is bounded, it also converges in $L^q(\Omega)$. Finally, continuity of $\|\cdot\|^q_{L^q}$ and weak lower semicontinuity of $\mathcal{R}$ gives
$$ \mathcal{T}_\alpha(\bar u;I_1) \le \liminf_{k\to \infty}  \mathcal{T}_\alpha(u_k;I_1) \le M.$$
Thus we have shown that the sublevel sets of $\mathcal{T}_\alpha(\cdot;I_1)$ are weakly sequentially closed and that $K$ is weak-strong sequentially continuous there.
\end{proof}

As already pointed out in Remark \ref{rmk:applicability}, while Theorem \ref{thm:convrates} is in principle applicable now, in most cases, due to the nonconvexity of $\mathcal{R}$, it is unlikely to actually apply. Below, we use the idea from Remark \ref{rmk:perfect} to construct an instance of a registration problem where Theorem \ref{thm:convrates} does apply however.

\begin{exmp}\label{exmp:main} Let $n=2$. Assume that $I_2$ is a rotated version of the exact data $I_1^\dagger$. That is, there is a deformation $u_R$, given by $u_R(x) = Rx$ for some $R \in SO(2)$, such that $\|I_2 \circ u_R - I^\dagger_1\|_{L^q(\Omega)} = 0$. Of course, $u_R$ must lie in $\mathcal{D}(K)$, which in this case translates to $\Omega$ being invariant with respect to the rotation $R$. We now construct a regularization functional $\mathcal{R}$ which not only satisfies all requirements from Proposition \ref{thm:registration} but which is also minimal for rotations. It then follows that $0\in \partial \mathcal{R}(u_R)$ and Theorem \ref{thm:convrates} applies.

For $u \in W^{1,p}(\Omega,\mathbb{R}^2)$, $p>2$, we define
$$ \mathcal{R}(u) = \int_\Omega f(\nabla u(x))\, dx,$$
where
$$f(A) = \mathrm{tr}\Big[(A^\top A)^{p/2}\Big] + pe^{1-\det A}$$
for all $A \in \mathbb{R}^{2 \times 2}$. This particular choice of integrand is inspired by the tangential distortion energy from \cite{IglRumSch15_report}. The identity deformation lies in the set $\mathrm{dom}\,\mathcal{R} \cap \mathcal{D}(K)$. Hence it is nonempty. That $f$ is not convex is easy to see as well. The coercivity estimate \eqref{eq:coercivity} follows from
$$ f(A) \ge \mathrm{tr}\Big[(A^\top A)^{p/2}\Big] = \lambda_1^p + \lambda_2^p \ge c(\lambda_1^2 + \lambda_2^2)^{p/2} = c |A|^p.$$
Here $\lambda_1, \lambda_2$ are the singular values of $A$ and $c>0$ is a constant whose existence is guaranteed by the equivalence of norms in finite dimensions. Convexity of the maps $x\mapsto e^{1-x}$ and $A\mapsto\lambda_1^p + \lambda_2^p$ (cf.\ \cite[Lemma 3.11]{Ped00}) yields polyconvexity of $f$. To verify minimality on $SO(2)$ it is convenient to rewrite $f$ in terms of its signed singular values $\mu_1=\mathrm{sgn}(\det A) \lambda_1$, $\mu_2=\lambda_2$:
$$ f(A) = |\mu_1|^p + \mu_2^p + p e^{1-\mu_1 \mu_2}.$$
Now minimality on $SO(2)$ translates to minimimality for $\mu_1=\mu_2=1$, which is easy to check.

The fact that $f$ only depends on signed singular values is actually equivalent to $f$ being $SO(2)\times SO(2)$ invariant, which is a desirable property in itself. See \cite[Sec.~5.3.3]{Dac08} for more details.
\end{exmp}

\section{Generalized Bregman distances for\\functionals with polyconvex integrands} \label{sec:main}
Following Grasmair's approach from \cite{Gra10b} we introduce an ``extended dual space" $W_{\!\mathrm{poly}}$ for the Sobolev space $W^{1,p}(\Omega,\mathbb{R}^N)$. The set $W_{\!\mathrm{poly}}$ is chosen such that we can prove a $W_{\!\mathrm{poly}}$-subdifferentiability result, similar to Lemma \ref{thm:subdif3}, for a certain class of functionals with polyconvex integrands.

\begin{dfn}
Let $\Omega \subset \mathbb{R}^n$ be open. For $p\in [1,\infty)$ set $U=W^{1,p}(\Omega,\mathbb{R}^N)$. Recall the notation from Section \ref{sec:polyconvex} and observe that for $u \in U$ by H\"{o}lder's inequality we have
\begin{equation}\label{eq:S}
	T_2(\nabla u) \in \prod_{s=2}^{N \wedge n} L^{\frac{p}{s}} (\Omega, \mathbb{R}^{\sigma(s)}) \eqqcolon S_2.
\end{equation}
Therefore, we let $W_{\!\mathrm{poly}}$ be the set of all functions $w : U \to \mathbb{R}$ for which there is a pair $(u^*, v^*) \in U^* \times S_2^*$ such that
\begin{equation*}
	w(u) = \langle u^*,u\rangle_{U^*,U} + \langle v^*, T_2(\nabla u) \rangle_{S_2^*,S_2}
\end{equation*}
for all $u\in U$.
\end{dfn}
\begin{rmk}
The set $W_{\!\mathrm{poly}}$ is simply the dual of $U \times S_2$ with its elements acting nonlinearly on $U$ instead of linearly on $U\times S_2$.
\end{rmk}
\begin{rmk}
Identifying $u^* \in U^*$ with $(u^*,0) \in W_{\!\mathrm{poly}}$ we can regard $W_{\!\mathrm{poly}}$ as a superset of $U^*$. Hence the generalized subdifferential
$$\partial_{\mathrm{poly}} \mathcal{R}(u) =
\begin{cases}
		\{w \in W_{\!\mathrm{poly}} : \mathcal{R}(v) \ge \mathcal{R}(u) + w(v) - w(u) \text{ for all } v \in U\}, & \mathcal{R}(u) \in \mathbb{R} \\
		\emptyset, & \mathcal{R}(u) \notin \mathbb{R}.
	\end{cases}
$$
can be regarded as a superset of the classical one $\partial \mathcal{R}(u)$ for all functionals $\mathcal{R}:U \to \mathbb{R}\cup\{\pm\infty\}$ and $u\in U$.
\end{rmk}
\begin{rmk}
Let $\bar u \in \mathrm{dom}\, \mathcal{R}$, $u\in U$ and $w=(u^*,v^*)\in \partial_{\mathrm{poly}} \mathcal{R}(\bar u)$. The associated $W_{\!\mathrm{poly}}$-Bregman distance can be written as
\begin{align*}
	D_w^{\mathrm{poly}}(u;\bar u) 
		&= \mathcal{R}(u) - \mathcal{R}(\bar u) - w(u) + w(\bar u) \\
		&= \mathcal{R}(u) - \mathcal{R}(\bar u) - \langle u^*, u - \bar u \rangle_{U^*,U} - \langle v^*, T_2(\nabla u) + T_2(\nabla \bar u) \rangle_{S_2^*,S_2}.
\end{align*}
Note that the first three terms in the second line correspond to the classical Bregman distance at $(\bar u, u^*)$, but their sum can be negative now, since $u^* \notin \partial \mathcal{R}(\bar u)$ in general.
\end{rmk}
The following statement justifies our definition of $W_{\!\mathrm{poly}}$ and is analogous to Lemma \ref{thm:subdif3}.
\begin{lem} \label{thm:wsubdif}
Let $\Omega \subset \mathbb{R}^n$ be an open set and let
	$$F:\Omega \times \mathbb{R}^N \times \mathbb{R}^{\tau(N, n)} \to \mathbb{R} \cup \{+ \infty\}$$
be a Carath\'eodory function. Assume that, for almost every $x\in \Omega$, the map $(u,\xi) \mapsto F(x,u,\xi)$ is convex and differentiable throughout its effective domain. Let $p\in[1,\infty)$ and define the following functional on $U =  W^{1,p}(\Omega,\mathbb{R}^N)$
	$$ \mathcal{R}(u) = \int_\Omega F(x,u(x),T(\nabla u(x))) \, dx.$$
If $\mathcal{R}(\bar{v}) \in \mathbb{R}$ and the function $x \mapsto \nabla_{u,\xi} F (x,\bar{v}(x),T(\nabla \bar{v}(x)))$ lies in
\begin{equation}\label{eq:Sspace}
	L^{p^*}(\Omega,\mathbb{R}^N) \times \prod_{s=1}^{N \wedge n} L^{(\frac{p}{s})^*} (\Omega, \mathbb{R}^{\sigma(s)}),
\end{equation}
then this function is a $W_{\!\mathrm{poly}}$-subgradient of $\mathcal{R}$ at $\bar v$.
\end{lem}
\begin{proof}
For almost every $x \in \Omega$ the map $(u,\xi) \mapsto F(x,u,\xi)$ is subdifferentiable throughout its effective domain. Therefore, for every $(v,\zeta) \in \mathrm{dom}\,F(x,\cdot,\cdot)$ we have
	\begin{align*}
		F(x,w,\eta)	&\ge F(x,v,\zeta) + \nabla_{u,\xi} F (x,v,\zeta) \cdot (w-v,\eta-\zeta) \\
					&= 	F(x,v,\zeta) + \nabla_{u} F (x,v,\zeta) \cdot (w-v) + \nabla_{\xi} F (x,v,\zeta) \cdot (\eta-\zeta)
	\end{align*}
for all $(w,\eta) \in \mathbb{R}^{N}\times \mathbb{R}^{\tau(N,n)}$. In particular, for functions $\bar{v},v\in U$, $\mathcal{R}(\bar{v}) \in \mathbb{R}$, we get
	\begin{align*}
		F(x,v(x),T(\nabla v(x))) &\ge F(x,\bar{v}(x),T(\nabla \bar{v} (x)))  \\
			& \quad {}+ \nabla_{u} F (x,\bar{v}(x),T(\nabla \bar{v}(x))) \cdot (v(x)-\bar{v}(x)) \\
		 	& \quad {}+ \nabla_\xi F (x,\bar{v}(x),T(\nabla \bar{v} (x))) \cdot (T(\nabla v(x))-T(\nabla \bar{v}(x)))
	\end{align*}
for almost every $x \in \Omega$. Integration over $\Omega$ gives
	\begin{align*}
	\mathcal{R}(v) &\ge \mathcal{R}(\bar{v}) + \int_\Omega \big[ \nabla_{u} F (x,\bar{v}(x),T(\nabla \bar{v} (x))) \cdot (v(x)-\bar{v}(x)) \\
		& \quad{}+ \nabla_{\xi} F (x,\bar{v}(x),T(\nabla \bar{v} (x))) \cdot (T(\nabla v(x))-T(\nabla \bar{v}(x)))\, \big] dx.
	\end{align*}
The integral on the right hand side is well-defined, if the functions
\begin{align*}
	x &\mapsto \nabla_{u} F (x,\bar{v}(x),T(\nabla \bar{v}(x))) \quad \text{and} \\
	x &\mapsto \nabla_{\xi} F (x,\bar{v}(x),T(\nabla \bar{v}(x)))
\end{align*}
lie in the right Lebesgue spaces. Considering that $v-\bar{v} \in L^p$ and
	$$T(\nabla v)-T(\nabla \bar{v}) \in \prod_{s=1}^{N \wedge n} L^{\frac{p}{s}} (\Omega, \mathbb{R}^{\sigma(s)}),$$
the right spaces are just given by \eqref{eq:Sspace}.
\end{proof}

\begin{exmp}
Let $N=n=2$. Then $T(A) = (A, \det A)$ for $A\in \mathbb{R}^{2 \times 2}$. Define an integrand by $F(x,u,A,\det A) = F(\det A)= (\det A)^2$. If $p=4$, then for all $u \in U = W^{1,4}(\Omega,\mathbb{R}^2)$ we have
	$$ \mathcal{R}(u) = \int_\Omega (\det \nabla u(x) )^2 \, dx \in \mathbb{R}$$
and the function $x\mapsto F'(\det \nabla u(x)) = 2 \det \nabla u(x) $ lies in $L^2(\Omega)$. By Lemma \ref{thm:wsubdif} functional $\mathcal{R}$ is $W_{\!\mathrm{poly}}$-subdifferentiable everywhere.
\end{exmp}

\begin{exmp}
Let $p > N = n \ge 2$, $q > 1$, and let $\Omega \subset \mathbb{R}^n$ be bounded. Consider the integrand $F(x,u,T(A)) = F(A, \det A) = |A|^p/p + |\det A|^q/q$. If $\bar v \in W^{1,\infty}(\Omega,\mathbb{R}^n)$, then clearly $\mathcal{R}(\bar v) \in \mathbb{R}$. In addition
	$$x\mapsto \nabla_\xi F(\nabla \bar v (x), \det \nabla \bar v (x)) = (|\nabla \bar v (x)|^{p-2}\nabla \bar v (x), |\det \nabla \bar v (x)|^{q-2}\det \nabla \bar v (x))$$
lies in $L^\infty$. Therefore, $\mathcal{R}$ has a $W_{\!\mathrm{poly}}$-subgradient everywhere on $W^{1,\infty}(\Omega,\mathbb{R}^n) \subset U = W^{1,p}(\Omega,\mathbb{R}^n)$, which implies that the associated $W_{\!\mathrm{poly}}$-Bregman distance is defined on a dense subset of $U$. In addition, the functional satisfies the coercivity estimate \eqref{eq:coercivity} and is weakly lower semicontinuous in $W^{1,p}$ according to Lemma \ref{thm:lsc}. Thus, at least from a theoretical perspective, the functional $\mathcal{R}$ is well-suited for regularizing inverse problems with $\mathbb{R}^n$-valued unknowns. Since $\mathcal{R}$ is not convex, however, it is not covered by most of existing regularization theory \cite{SchGraGroHalLen09,SchuKalHofKaz12}.
\end{exmp}

The next theorem shows that standard convergence rates results can be carried over to the $W_{\!\mathrm{poly}}$-Bregman distance. Its proof is analogous to that of Theorem \ref{thm:convrates} and therefore omitted.

\begin{thm}\label{thm:polyconvrates}
	Let the assumptions of Theorem \ref{thm:welldef} hold with $U = W^{1,p}(\Omega,\mathbb{R}^N)$. In addition assume that $\mathcal{R}$ has a $W_{\!\mathrm{poly}}$-subgradient $w$ at an $\mathcal{R}$-minimizing solution $u^\dagger$ and that there are constants $\beta_1 \in [0,1)$ and $\bar\alpha, \beta_2, \rho > 0$ satisfying $\bar\alpha\mathcal{R}(u^\dagger) < \rho$ such that
	\begin{equation*}
		w(u^\dagger) - w(u) \le \beta_1 D^{\mathrm{poly}}_w(u;u^\dagger) + \beta_2 \| K(u) - v^\dagger\|
	\end{equation*}
	holds for all $u$ with $\mathcal{T}_{\bar \alpha} (u;v^\dagger) \le \rho$.
	
	If $p>1$, assume $\alpha(\delta) \sim \delta^{p-1}$. Then
	\begin{equation*}
		D^{\mathrm{poly}}_w (u^\delta_{\alpha};u^\dagger) = O(\delta) \quad \text{and} \quad \|K(u^\delta_{\alpha}) - v^\delta\| 	= O(\delta).
	\end{equation*}
	
	If $p=1$, assume $\alpha(\delta) \sim \delta^\epsilon$ for $\epsilon \in [0,1)$. If $\epsilon = 0$, additionally assume that $0 < \alpha(\delta) \beta_2 < 1$. Then
	\begin{equation*}
		D^{\mathrm{poly}}_w(u^\delta_{\alpha};u^\dagger) = O(\delta^{1-\epsilon}) \quad \text{and} \quad \|K(u^\delta_{\alpha}) - v^\delta\| 	= O(\delta).
	\end{equation*}
\end{thm}

\begin{rmk}
Note that Theorem \ref{thm:polyconvrates} does not require $\mathcal{R}$ to have a polyconvex integrand, just as Theorem \ref{thm:convrates} does not require $\mathcal{R}$ to be convex.
\end{rmk}

\begin{rmk}
Another strategy for extending Bregman distances to functionals with polyconvex integrands could consist in lifting the functionals from $U$ to a product space $\tilde U$: Let $F$ be an integrand as in Lemma \ref{thm:lsc} and define
$$\mathcal{R}(u) = \int_\Omega F(x,u(x),T(\nabla u(x)))\,dx$$
on $U$. Now we can, for example define the lifted functional
$$\tilde{\mathcal{R}}(u,v) = \int_\Omega F(x,u(x),(\nabla u(x), v(x)))\,dx$$
on $\tilde U = U \times S_2$. It satisfies $\tilde{\mathcal{R}}(u,T_2(\nabla u)) = \mathcal{R}(u)$ for all $u\in U$. Adapting source conditions \eqref{eq:sourcecon} we could measure convergence rates with $D_{\tilde{u}^*}(u^\delta_\alpha; u^\dagger)$ where $\tilde{u}^* \in \partial \tilde{\mathcal{R}}(u^\dagger,T_2(\nabla u^\dagger))$. The following inclusion is straightforward
$$\partial \tilde{\mathcal{R}}(u,T_2(\nabla u)) \subset \partial_{\mathrm{poly}} \mathcal{R}(u).$$
In contrast to $\partial_{\mathrm{poly}} \mathcal{R}(u)$, however, there is no obvious way in which $\partial \tilde{\mathcal{R}}(u,T_2(\nabla u))$ could be viewed as an extension of $\partial \mathcal{R}(u)$.
\end{rmk}

\section{Conclusion}
Convexity is an unnecessarily strong requirement for functionals $\mathcal{R}$ defined on $W^{1,p}(\Omega,\mathbb{R}^N)$, if the main concern is to ensure weak lower semicontinuity. In fact, polyconvexity of the integrand, or even quasiconvexity, is enough. However, if $\mathcal{R}$ is supposed to serve as a regularization functional, then the problem is how to measure convergence rates. The standard approach using classical Bregman distances $D_{u^*}(u_\alpha^\delta;u^\dagger)$ must be expected to fail in general due to the lack of convexity. In this article we have tried to answer two questions. (i) Are there instances of nonconvex variational regularization where standard convergence rates results \emph{do} apply? (ii) What could a general strategy for obtaining convergence rates for regularization functionals with polyconvex integrands look like?

With Example \ref{exmp:main} we have given a positive answer to the first question. It is based on the fact, explained in Remark \ref{rmk:perfect}, that source conditions are automatically satisfied, if $\mathcal{R}$ has a minimizer which is also an exact solution of the operator equation \eqref{eq:ip}. Exploiting the fact the polyconvexity is compatible with minimality on $SO(n)$ we constructed an instance of the image registration problem with nonconvex regularization where a standard convergence rates result as given in Theorem \ref{thm:convrates} applies.

The second question was addressed by introducing $W_{\!\mathrm{poly}}$-Bregman distances, which are based on a recent idea from \cite{Gra10b} and which have a reasonably large domain of definition for a certain class of functionals with polyconvex integrands (cf.~Lemma \ref{thm:wsubdif}). By adapting the usual source conditions one can obtain linear convergence rates also for the $W_{\!\mathrm{poly}}$-Bregman distances (see Theorem \ref{thm:polyconvrates}).

\subsection*{Acknowledgements}
We thank Jos\'{e} A.~Iglesias for helpful comments and discussions, in particular regarding Example \ref{exmp:main}. Both authors acknowledge support by the Austrian Science Fund (FWF) within the national research network Geometry and Simulation, project S11704 (Variational Methods for Imaging on Manifolds). In addition, the work of OS is supported by the Austrian Science Fund (FWF), project P26687-N25
Interdisciplinary Coupled Physics Imaging. 

\def\cprime{$'$}
  \providecommand{\noopsort}[1]{}\def\ocirc#1{\ifmmode\setbox0=\hbox{$#1$}\dimen0=\ht0
  \advance\dimen0 by1pt\rlap{\hbox to\wd0{\hss\raise\dimen0
  \hbox{\hskip.2em$\scriptscriptstyle\circ$}\hss}}#1\else {\accent"17 #1}\fi}

\end{document}